\documentclass{elsarticle}
\usepackage{verbatim}
\usepackage{lineno,hyperref}
\usepackage{amssymb,amsmath,amsthm, amsfonts, mathrsfs, bbm, dsfont}
\usepackage[utf8]{inputenc}
\usepackage{bbm}
\usepackage{amssymb,amsmath,amsthm}
\usepackage{fancyhdr}
\usepackage{amssymb}
\usepackage{fancyref}
\usepackage{mathrsfs}
\usepackage{graphics}
\usepackage{tikz}
\usepackage{stmaryrd}
\usepackage{bm}
 \usepackage{mathrsfs}
\usepackage[all]{xy}
\usepackage{amscd}
\usepackage{setspace}
\usepackage{tikz}
\usepackage{amssymb,amsmath,amsthm}
\usepackage{fancyhdr}
\usepackage{amssymb}
\usepackage{fancyref}
\usepackage{mathrsfs}
\usepackage{graphics}
\usepackage{stmaryrd}
\usepackage{bm}
 \usepackage{mathrsfs}
\usepackage[all]{xy}
\usepackage{amscd}
\usepackage{tikz-cd}
\usepackage{braids}
\usepackage{extarrows}
\modulolinenumbers[5]
\usepackage[top=1in, bottom=1in, left=1.25in, right=1.25in]{geometry}
\newtheorem{theorem}{Theorem}[section]
\newtheorem{lemma}{Lemma}[section]

\newtheorem{corollary}{Corollary}[section]

\newtheorem{definition}{Definition}[section]
\newtheorem{example}{Example}[section]
\newtheorem{remark}{Remark}

\usepackage{extarrows}

\modulolinenumbers[5]
\begin{document}

\begin{frontmatter}

\title{Compact operators under Orlicz functions\tnoteref{mytitlenote}}


\author[mymainaddress,mysecondaryaddress]{Ma Zhenhua \corref{Ma Zhenhua}}
\cortext[Ma Zhenhua]{Corresponding author}
\ead{mazhenghua\_1981@163.com}

\author[mymainaddress]{Ji Kui}
\ead{jikui@hebtu.edu.cn}

\author[mymainaddress]{Li Yucheng}
\ead{liyucheng2004@126.com}
\address[mymainaddress]{Postdoctoral Research Station, Hebei Normal University, Shiajiazhuang,  050024, P. R. China}
\address[mysecondaryaddress]{School of Mathematics and Physics, Hebei University of Architecture, Zhangjiakou, 075024, P. R. China}

\begin{abstract}
The purpose of this paper is to research the compact operators under  Orlicz functions. Firstly, the definition of noncommutative Orlicz sequence spaces (denoted by $S_{\varphi}(\mathbb{H})$) is given, these spaces generalize the Schatten classes $S_{p}(\mathbb{H})$. After some relations of trace and norm on this spaces have been obtained, one give the criterion of reflexivity of these spaces. Finally, as an application, we find the Toeplitz operator $x_{1-|z|^{2}}$ on the Bergman space $L_{a}^{2}(\mathbb{D})$ belongs to some $S_{\varphi}(\mathbb{H})$, hence the trace and the norm of $x_{1-|z|^{2}}$ could be computed.
\end{abstract}


\end{frontmatter}

\linenumbers
\section{Introduction and Preliminaries}
Let $\mathcal{B}(\mathbb{H})$ denotes the algebra of the bounded operators acting on a complex and separable Hilbert space $\mathbb{H}$ which is endowed with an inner product by $\langle\cdot,\cdot\rangle$, and $x\in \mathcal{B}(\mathbb{H})$ is a linear compact operator. If we denote by $x^{\ast}: \mathbb{H}\rightarrow \mathbb{H}$ the adjoint of $x$, then the linear operator $x^{\ast}x: \mathbb{H}\rightarrow \mathbb{H}$ is positive and compact. Let $\{e_{k}\}_{k}$ be an orthonormal  basis for $\mathbb{H}$ consisting of eigenvectors of $|x|=(x^{\ast}x)^{\frac{1}{2}}$ and $s_{k}(x)$ the eigenvalue (decreasingly ordered) corresponding to the eigenvector $\{e_{k}\}_{k}, k=1,2,\ldots$. The non-negative numbers $s_{k}(x), k=1,2,\ldots$ are called the singular values of $x:\mathbb{H}\rightarrow \mathbb{H}$.

If $\sum\limits^{\infty}_{k=1}s_{k}(x)<\infty,$ the  operator $x:\mathbb{H}\rightarrow \mathbb{H}$ is called a trace class operator. The set of all trace class operators is denoted by $S_{1}(\mathbb{H})$. It can be shown that $S_{1}(\mathbb{H})$ is a Banach space in which the trace norm $\|\cdot\|_1$ is given by $\|x\|_{1}=\sum^{\infty}\limits_{k=1}s_{k}(x)$.

On the other side, if $x:\mathbb{H}\rightarrow \mathbb{H}$ is an operator in $S_{1}(\mathbb{H})$ and  $\{e_{k}\}_{k}, k=1,2,\ldots$ is any orthonormal basis for $\mathbb{H}$. Then, the series $\sum\limits^{\infty}_{k=1}\langle xe_{k},e_{k}\rangle$ is absolutely convergent and the sum is independent of the choice of the orthonormal basis $e_{k}$. Thus, we can define the trace ${\rm Tr}(x)$ of any linear operator $x:\mathbb{H}\rightarrow \mathbb{H}$ in $S_{1}(\mathbb{H})$ by
$${\rm Tr}(x):= \sum^{\infty}_{k=1}\langle e_{k},xe_{k}\rangle.$$
 By the spectral decomposition theorem of compact operators, for any $\xi\in \mathbb{H}$ there exist orthonormal sequence $\xi_{n}\in \mathbb{H}$ such that,
$$x\xi=\sum_{n\geq0}s_{n}(x)\langle\xi_{n},\xi\rangle\xi_{n},$$
hence one get
$${\rm Tr}(x)= \sum_{n\geq0}s_{n}(x).$$
For more details of $S_{1}(\mathbb{H})$ please refer to \cite{GK,Murphy}.

Similar to the definition of $S_{1}(\mathbb{H})$, the  Schatten classes $S_{p}(\mathbb{H})$ is given as follows:
for any $0<p<\infty$, let $x$ be a compact operator of $\mathbb{H}$,
$$S_{p}(\mathbb{H})=\left\{x: \ {\rm Tr}\left(|x|^{p}\right)=\sum\limits_{n\geq0}s_{n}\left(|x|^{p}\right)=\sum\limits_{n\geq0}\left(s_{n}\left(|x|\right)\right)^{p}<\infty\right\},$$
for any $x\in S_{p}(\mathbb{H})$ one can give the norm as
$$\|x\|_{p}=\left(\sum\limits_{n\geq0}\left(s_{n}\left(|x|\right)\right)^{p}\right)^{\frac{1}{p}}.$$
Especially, for $p=2$ the $S_{2}(\mathbb{H})$ is called Hilbert-Schmidt operator space \cite{Bridges}, and the norm is $$\|x\|_{2}=\left(\sum\limits_{n\geq0}\left(s_{n}\left(|x|\right)\right)^{2}\right)^{\frac{1}{2}}.$$ Literatures on the Schatten class operators and their applications are very rich, please see the very interesting recent papers \cite{Bo,WKG,MI1,MI2,SVD}.
and references cited therein. Especially,  Gil',M.I. extent  some
useful results on determinants of Schatten–von Neumann operators to the operators with the Orlicz type norms and modular function, respectively \cite{MI2,MI4}. The purpose of this paper is to research the compact operators under Orlicz functions which named noncommutative Orlicz  sequence spaces, and many results of these spaces could extent some useful results of the Schatten class $S_{p}(\mathbb{H})$, such as the H$\ddot{{\rm o}}$lder Inequality and so on. Besides, some applications are given to support our theory, such as Toeplitz operator on the Bergman space.

A convex function $\varphi: [0,\infty)\rightarrow[0,\infty]$ is called an Orlicz function
if it is nondecreasing and continuous for $\alpha\geq0$ and such that $\varphi(0)=0,\,\varphi(\alpha)>0$ and $\varphi(\alpha)\rightarrow\infty$ as $\alpha\rightarrow\infty$.

Let $\varphi$ be an Orlicz function and  $l^{0}$ the set of all real sequence. We define the Orlicz sequence space as follows
$$l_{\varphi}=\left\{x\in l^{0}: I_{\varphi}(\lambda x(i))=\sum^{\infty}_{i=1}\varphi\left(\lambda x(i)\right)<\infty,  \exists \lambda>0\right\}.$$
It is well known that $l_{\varphi}$ equipped with so called the Luxemburg norm
$$\|x\|=\inf\left\{\lambda>0: I_{\varphi}\left( \frac{x(i)}{\lambda}\right)\leq1\right\}$$
is a Banach space.

Further we say an Orlicz function $\varphi$ satisfies condition $\delta_{2}$ near zero, shortly $\varphi\in\delta_{2}(0)$, if there exist constant $u_{0}>0$  and $k>2$ such that $\varphi(2u)\leq k\varphi(u)$ for all $|u|\leq u_{0}$. For the background of Orlicz functions and Orlicz spaces please see \cite{Chen,Rao}.
\section{Noncommutative Orlicz sequence spaces}
With the help of Orlicz function and functional calculus, if $\varphi$ is a continuous function, we have $${\rm Tr}(\varphi(x))= \sum\limits^{\infty}_{k=1}\langle e_{k},\varphi(x)e_{k}\rangle=\sum\limits_{n\geq0}\varphi(s_{n}(x))=\sum\limits_{n\geq0}s_{n}(\varphi (x)).$$ Hence, if $\varphi$ is an Orlicz function and $x$ a compact operator of $\mathbb{H}$, the definitions of the noncommutative Orlicz sequence spaces and its a useful subspace could be given.

Hence, in this section, after the definitions of noncommutative Orlicz sequence spaces and its a subspace which first appear in \cite{MI3} are given, and we find these spaces could generalize the $S_{p}(\mathbb{H})$ and the classical Orlicz sequence spaces respectively. At the end of this section, the Orlicz norm of the rank-one operator $e_{n}\otimes e_{n}$ is given.
\begin{definition}
If $x\in \mathcal{B}(\mathbb{H})$ is a compact operator, then the noncommutative Orlicz sequence spaces $S_{\varphi}(\mathbb{H})$ and its a subspace $E_{\varphi}(\mathbb{H})$ are defined as follows:
$$S_{\varphi}(\mathbb{H})=\left\{x: {\rm Tr}\left(\varphi(\lambda x)\right)=\sum_{n\geq0}\varphi(\lambda s_{n}(|x|))<\infty ,  \exists\lambda>0\right\},$$
$$E_{\varphi}(\mathbb{H})=\left\{x: {\rm Tr}\left(\varphi(\lambda x)\right)=\sum_{n\geq0}\varphi(\lambda s_{n}(|x|))<\infty , \forall \lambda>0\right\}.$$
where  $\rm {Tr}$ is the usual trace functional of $\mathcal{B}(\mathbb{H})$.
\end{definition}

For convenience, we use $S_{\varphi}, E_{\varphi}$ denote $S_{\varphi}(\mathbb{H})$ and $E_{\varphi}(\mathbb{H})$ respectively, and it could be equipped with the Luxemburg norm as
\begin{eqnarray*}
\|x\|_{\varphi}&=&\inf\left\{\lambda>0, {\rm Tr}\left(\varphi\left( \frac{x}{\lambda}\right)\right)\leq1\right\}\\
&=& \inf\left\{\lambda>0, \sum_{n\geq0}\varphi
\left( s_{n}\left(\frac{|x|}{\lambda}\right)\right)\leq1\right\}\\
&=& \inf\left\{\lambda>0, \sum_{n\geq0}
\varphi\left(\frac{s_{n}\left(|x|\right)}{\lambda}\right)\leq1\right\}\\
&=&\|s(x;n)\|,
\end{eqnarray*}
where $s(x;n)$ is the singular value sequence of $x$ and $\|\cdot\|$ is the Luxemburg norm of  classical Orlicz space. By the theory of classical Orlicz sequence, it is easy to know $(S_{\varphi}, \|x\|_{\varphi})$  and $(E_{\varphi}, \|x\|_{\varphi})$ are Banach spaces.

Also, one can define another norm which named Orlicz norm on  $S_{\varphi}$ as follows
$$\|x\|_{\varphi}^{o}=\sup\left\{{\rm Tr}(|xy|):  {\rm Tr}(\psi(y))\leq1\right\},$$
where $x\in S_{\varphi}$ and $\psi: [0,\infty)\rightarrow [0,\infty] $ is defined by $\psi(u)=\sup\{uv-\varphi(v): v\geq 0\}$. Here we call $\psi$ the complementary function of $\varphi$. Similar the classical case, one also could give the another equivalent definition of the Orlciz norm which named Amemiya norm, $\|x\|_{\varphi}^{A}=\inf\limits_{k>0} \frac{1}{k}\left({1+\rm Tr}(\varphi(kx)\right)$, from \cite{Hudzik} it is easy to know, $\|x\|_{\varphi}^{o}=\|x\|_{\varphi}^{A}.$

By the classic Orlicz theory we also have the following relation of these  norms
$$\|x\|_{\varphi}\leq\|x\|_{\varphi}^{o}\leq2\|x\|_{\varphi}.$$
Generally speaking, from the definitions it is easy to know $E_{\varphi}\subsetneqq S_{\varphi}$, but if $\varphi\in \delta_{2}(0)$ one could get the following theorem,
\begin{theorem}
$S_{\varphi}=E_{\varphi}$ if and only if $\varphi\in \delta_{2}(0)$.
\end{theorem}
\begin{proof}
If $\varphi\in \delta_{2}(0)$. It is obvious true $E_{\varphi}\subseteq S_{\varphi}$ by the definition. Hence we only need to prove $S_{\varphi}\subseteq E_{\varphi}.$

For any $x\in S_{\varphi}$, there exists a $\lambda_{0}>0$ such that
 $${\rm Tr}\left(\varphi(\lambda_{0} x)\right)=\sum_{n\geq0}\varphi(\lambda_{0} s_{n}(|x|))<\infty,$$ then for any $\lambda>{0}$, one could get
 $${\rm Tr}\left(\varphi(\lambda x)\right)=\sum_{n\geq0}\varphi(\lambda s_{n}(|x|))<\infty.$$
In fact, if $\frac{\lambda}{\lambda_{0}}\leq1$, by convexity of $\varphi$,
$${\rm Tr}\left(\varphi(\lambda x)\right)=\sum_{n\geq0}\varphi\left(\frac{\lambda}{\lambda_{0}}\lambda_{0} s_{n}(|x|)\right)\leq\frac{\lambda}{\lambda_{0}}\sum_{n\geq0}\varphi\left(\lambda_{0} s_{n}(|x|)\right)<\infty.$$
If $\frac{\lambda}{\lambda_{0}}>1$, define $n_{0}=\sup\{n,s_{n}(|x|)\leq u_{0}\}$, where $u_{0}\in (0,s_{1})$. Since $\varphi\in \delta_{2}(0)$, we have
\begin{eqnarray*}
{\rm Tr}\left(\varphi(\lambda x)\right)&=&\sum_{n\geq0}\varphi \left(\lambda s_{n}(|x|)\right)\\
&=& \sum_{n=1}^{n_{0}}\varphi \left( \lambda s_{n}(|x|)\right)+\sum_{n=n_{0}+1}^{\infty}\varphi \left( \frac{\lambda}{\lambda_{0}}\lambda_{0}s_{n}(|x|)\right)\\
&\leq&n_{0}\varphi(\lambda s_{1}(|x|))+k\sum_{n=n_{0}+1}^{\infty}\varphi \left( \lambda_{0}s_{n}(|x|)\right)\\
&\leq&n_{0}\varphi(\lambda s_{1}(|x|))+k {\rm Tr}(\varphi(\lambda_{0}x))\\
&<&+\infty.
\end{eqnarray*}

Now, if $S_{\varphi}=E_{\varphi}$, similar to the Example 1.19 of \cite{Chen} we have $\varphi\in \delta_{2}(0)$. This completes the proof.
\end{proof}
\begin{remark}
$(1)$ In the case $\varphi(x)=|x|^{p},\,\,1\leq p<\infty$ for any compact operator $x\in \mathcal{B}(\mathbb{H})$,  $S_{\varphi}$ is nothing but the Schatten
classes $S_{p}$ and the Luxemburg norm generated by this function is expressed by the formula
$$\|x\|_{p}={\rm Tr}(|x|^{p})=\sum_{n\geq0}\left(( s_{n}(|x|))^{p}\right)^{\frac{1}{p}}.$$

 Especially, if $p=1$, $S_{\varphi}=S_{1}$ and if $p=2$, $S_{\varphi}$ is the Hilbert-Schmidt operator $S_{2}$ which satisfies $\varphi\in\delta_{2}(0)$.

$(2)$ If $\varphi\in\delta_{2}(0)$, let $x: l_{2}(\mathbb{N})\rightarrow l_{2}(\mathbb{N})$ with $xe_{n}=a_{n}e_{n}$, where $e_{n}$ is the orthonormal basis of $l_{2}(\mathbb{N})$, then it is obvious that $x\in S_{\varphi}$ if and only if $a=\{a_{n}\}\in l_{\varphi}$ and $\|x\|_{\varphi}=\|a\|$.
\end{remark}

Let $\mu(n)$ denotes the multiplicity of singular value $s_{n}(x)$. The following theorem gives the Orlicz norm of rank one operator $e_{ n}\otimes e_{n}$.
\begin{theorem}
For each $y\in S_{\psi}$ is positive and ${\rm Tr}(\psi(y))\leq1$, where  $\psi(x)$ is invertible on $\mathbb{R}^{+}$, and $\psi(x)\in \delta_{2}(0)$, Let $\{e_{n}\}$  be a orthonormal bases, then  $$\|e_{ n}\otimes e_{n}\|^{o}_{\varphi}=\psi^{-1}\left(\frac{1}{\mu(n)}\right)\mu(n),$$
where $e\otimes h$ is rank one operator (a bounded operator) on $\mathbb{H}\rightarrow \mathbb{H}$ with $\xi\mapsto\langle h,\xi\rangle e$.
\end{theorem}
\begin{proof}
First, by the Jensen Inequality,
$$\psi(s_{n}(y))=\psi\left(\frac{1}{\mu(n)}\sum_{\mu(n)}s_{n}(y)\right)\leq \frac{1}{\mu(n)}{\rm Tr}(\psi(y))\leq \frac{1}{\mu(n)}.$$
Using spectral decomposition of compact operators, we have $y=\sum\limits_{n\geq0}s_{n}(y)e_{ n}\otimes e_{n}$.
Therefore,
$$\|e_{ n}\otimes e_{n}\|^{o}_{\varphi}=\sup\left\{\sum_{\mu(n)}s_{n}(y): {\rm Tr}(\psi(y))\leq 1\right\}\leq \psi^{-1}\left(\frac{1}{\mu(n)}\right)\mu(n).$$
On the other hand, observing that
$${\rm Tr}\left(\psi\left(\psi^{-1}\left(\frac{1}{\mu(n)}e_{ n}\otimes e_{n}\right)\right)\right)=\frac{1}{\mu(n)}{\rm Tr}(e_{ n}\otimes e_{n})\leq1,$$
we obtain
$$\|e_{ n}\otimes e_{n}\|^{o}_{\varphi}\geq\sum^{\infty}_{n=1}s_{n}\left(\psi^{-1}\left(\frac{1}{\mu(n)}e_{ n}\otimes h_{n}\right)e_{ n}\otimes e_{n}\right)=\psi^{-1}\left(\frac{1}{\mu(n)}\right)\mu(n).$$
Thus,
$$\|e_{ n}\otimes e_{n}\|^{o}_{\varphi}=\psi^{-1}\left(\frac{1}{\mu(n)}\right)\mu(n).$$
\end{proof}
\begin{corollary}
If $x\in S_{p}$ is positive and $\mu(n)$ is the multiplicity of singular value $s_{n}(x)$. Then $$\|e_{ n}\otimes e_{n}\|^{o}_{p}=\mu(n)^{1-p}.$$
\end{corollary}
\begin{corollary}
For a positive operator $x$, if $\varphi(x)=\frac{1}{\alpha}|x|^{\alpha}$ and $\alpha>1$. Then $$\|e_{ n}\otimes e_{n}\|^{o}_{\varphi}=\left(\frac{\beta}{\mu(n)}\right)^{\frac{1}{\beta}}\mu(n)$$
where $\frac{1}{\alpha}+\frac{1}{\beta}=1$.
\end{corollary}
\section{Trace and the Luxemburg norm}
In this section, we will give some relations of Trace and the Luxemburg norm, which generalize the correspondence results of $S_{p}$.
\begin{lemma}
If $x\in S_{\varphi}$ and $y\in S_{\psi}$, then we have

$(1)$ If $\|x\|_{\varphi}>1$,  then ${\rm Tr}(\varphi(x))\geq \|x\|_{\varphi};$

$(2)$ If $x\neq0$, then ${\rm Tr}\left(\varphi\left(\frac{x}{\|x\|_{\varphi}}\right)\right)\leq1;$

$(3)$ If $\|x\|_{\varphi}\leq1$,  then ${\rm Tr}(\varphi(x))\leq \|x\|_{\varphi};$

$(4)$ (H$\ddot{o}$lder Inequality) ${\rm Tr}(|xy|)\leq \|x\|^{o}_{\varphi}\|y\|_{\psi};$

$(5)$ $xy, yx \in S_{1}.$
\end{lemma}
\begin{proof}
$(1)$ If $\|x\|_{\varphi}>1$,  we choose $0<\varepsilon<1$ such that $(1-\varepsilon)\|x\|_{\varphi}>1$. Then
\begin{eqnarray*}
1&<&{\rm Tr}\left(\varphi\left(\frac{x}{(1-\varepsilon)\|x\|_{\varphi}}\right)\right)\\
&\leq&\frac{1}{(1-\varepsilon)\|x\|_{\varphi}}{\rm Tr}(\varphi(x))+\left(1-\frac{1}{(1-\varepsilon)\|x\|_{\varphi}}\right){\rm Tr}(\varphi(0))\\
&=&\frac{1}{(1-\varepsilon)\|x\|_{\varphi}}{\rm Tr}(\varphi(x)),
\end{eqnarray*}
which imply
$$(1-\varepsilon)\|x\|_{\varphi}<{\rm Tr}(\varphi(x)),$$
let ~$\varepsilon\rightarrow0$ we can get the conclusion.

$(2)$ For any $\varepsilon>0$, there exists $\lambda_{\varepsilon}>0$ such that ${\rm Tr}\left(\varphi\left(\frac{x}{\lambda_{\varepsilon}}\right)\right)\leq1$, by the definition of Luxemburg norm, if $\lambda_{\varepsilon}<\varepsilon+\|x\|_{\varphi}$
one get that ${\rm Tr}\left(\varphi\left(\frac{x}{\|x\|_{\varphi}+\varepsilon}\right)\right)\leq1$, by the arbitrariness of $\varepsilon$ we can get,
$${\rm Tr}\left(\varphi\left(\frac{x}{\|x\|_{\varphi}}\right)\right)\leq1.$$

$(3)$ If $x\in S_{\varphi}$ and $\|x\|_{\varphi}\leq1$, let $x\neq0$. By the definition of $\|x\|_{\varphi}$, there exists a decreasing sequence $\{\lambda_{n}\}$ which converges to $\|x\|_{\varphi}$  such that
$${\rm Tr}\left(\varphi\left(\frac{x}{\lambda_{n}}\right)\right)=\sum\limits_{n\geq0}\varphi\left(\frac{1}{\lambda_{n}}s_{n}(x)\right)\leq1.$$
By Levy's theorem,
$${\rm Tr}\left(\varphi\left(\frac{x}{\|x\|_{\varphi}}\right)\right)=\sum\limits_{n\geq0}\varphi\left(\frac{1}{\|x\|_{\varphi}}s_{n}(x)\right)\leq1.$$
Thanks for the convexity of $\varphi$ and $\|x\|_{\varphi}\leq1$,
$$\frac{1}{\|x\|_{\varphi}}{\rm Tr}(\varphi(x))\leq{\rm Tr}\left(\varphi\left(\frac{x}{\|x\|_{\varphi}}\right)\right)\leq1,$$
hence,
$${\rm Tr}(\varphi(x))=\sum\limits_{n\geq0}\varphi\left(s_{n}(x)\right)\leq \|x\|_{\varphi}.$$

$(4)$ By $(3)$ and the definition of the Orlicz norm we have
${\rm Tr}\left(\frac{|xy|}{\|y\|_{\psi}}\right)\leq\|x\|^{o}_{\varphi}$, which could get the conclusion.

$(5)$ \begin{eqnarray*}
{\rm Tr}(|xy|)&=&\sum^{\infty}_{n+m=0}s_{n+m+1}(|xy|)\\
&\leq&\sum^{\infty}_{n+m=0}\left(s_{n+1}(|x|)s_{m+1}(|y|)\right)\\
&\leq&\sum^{\infty}_{n+m=0}\left(\varphi\left(s_{n+1}(|x|)\right)+\psi\left(s_{m+1}(|y|)\right)\right)\\
&<&\infty.
\end{eqnarray*}
In the above we have used 2.13 of \cite{GK} and the Orlicz Young inequality for the first and the second inequalities, respectively. Hence, we have $xy\in S_{1}$. Using the same method, one can get $yx\in S_{1}$.
\end{proof}

\begin{theorem}
If $\varphi\in \delta_{2}(0)$, for any $x, y \in S_{\varphi}$ we have

$(1)$ ${\rm Tr}\left(\varphi\left(\frac{x}{\|x\|_{\varphi}}\right)\right)=1$;

$(2)$ $\|x\|_{\varphi}=1$ if and only if ${\rm Tr}(\varphi(x))=1$;

$(3)$ ${\rm Tr}(\varphi(x_{n}))\rightarrow0$ if and only if $\|x_{n}\|_{\varphi}\rightarrow0$;

$(4)$ ${\rm Tr}(\varphi(x+y))\leq \frac{k}{2} \left[{\rm Tr}(\varphi(x))+{\rm Tr}(\varphi(y))\right]$, where $k$ is the real number which satisfies for any $\alpha>0$ such that $\varphi(2\alpha)\leq k \varphi(\alpha)$.
\end{theorem}
\begin{proof}
$(1)$ Let $x\in S_{\varphi}$, by Proposition 0 of \cite{Anna},
\begin{eqnarray*}
{\rm Tr}\left(\varphi\left(\frac{x}{\|x\|_{\varphi}}\right)\right)&=&\sum_{n\geq0}\varphi\left(\frac{1}{\|x\|_{\varphi}}s_{n}(|x|)\right)\\
&=&\sum_{n\geq0}\varphi\left(\frac{1}{\|s(x;n)\|}s_{n}(|x|)\right)\\
&=&1.
\end{eqnarray*}
$(2)$ If $\|x\|_{\varphi}=1$, then obvious we have ${\rm Tr}(\varphi(x))=1$. If ~${\rm Tr}(\varphi(x))=1$, by remark 8.15 of \cite{Musielak} we can get the conclusion.

$(3)$ This can be get by the Theorem 8.14 of \cite{Musielak} .

$(4)$ We know there exist partially isometric operator $u,v$ such that $|x+y|\leq u|x|u^{*}+v|x|v^{*}$, by the convexity of $\varphi$,
\begin{eqnarray*}
{\rm Tr}(\varphi(x+y))&\leq& {\rm Tr}\left(\varphi\left(u|x|u^{*}+v|y|v^{*}\right)\right)\\
&=&{\rm Tr}\left(\varphi\left(2\frac{u|x|u^{*}+v|y|v^{*}}{2}\right)\right)\\
&\leq&\frac{k}{2}\left[{\rm Tr}\left(\varphi\left(u|x|u^{*}\right)\right)+{\rm Tr}\left(\varphi\left(v|y|v^{*}\right)\right)\right]\\
&=&\frac{k}{2}\left[\sum_{n\geq0}\varphi\left(s_{n}(u|x|u^{*})\right)+\sum_{n\geq0}\varphi\left(s_{n}(v|y|v^{*})\right)\right]\\
&\leq&\frac{k}{2}\left[\sum_{n\geq0}\varphi\left(s_{n}(|x|)\right)+\sum_{n\geq0}\varphi\left(s_{n}(|y|)\right)\right]\\
&=&\frac{k}{2}\left[{\rm Tr}(\varphi(x))+{\rm Tr}(\varphi(y))\right].
\end{eqnarray*}
\end{proof}
\begin{corollary}
$(1)$ ${\rm Tr}\left(\frac{x^{p}}{\|x\|_{\varphi}^{p}}\right)=1.$ Especially, if $x$ belongs to $S_{1}$ and $S_{2}$ respectively, we have ${\rm Tr}\left(\frac{x}{\|x\|_{1}^{1}}\right)=1$ and ${\rm Tr}\left(\frac{x^{2}}{\|x\|_{2}^{2}}\right)=1.$

$(2)$ If $x$ is a Schatten operator, then $\varphi(2x)=|2x|^{p}=2^{p}|x|^{p}$. Let $k=2^{p}$ we can get
$${\rm Tr}\left(|x+y|^{p}\right)\leq 2^{p-1}\left[{\rm Tr}\left(|x|^{p}\right)+{\rm Tr}\left(|y|^{p}\right)\right]$$
or
$$\sum_{n\geq0}s_{n}\left(|x+y|^{p}\right)\leq 2^{p-1}\left[\sum_{n\geq0}s_{n}\left(|x|^{p}\right)+\sum_{n\geq0}s_{n}\left(|y|^{p}\right)\right].$$
Especially, for $x\in S_{1}$, we have
$$\sum_{n\geq0}s_{n}\left(|x+y|\right)\leq \sum_{n\geq0}s_{n}\left(|x|\right)+\sum_{n\geq0}s_{n}\left(|y|\right);$$
For $x\in S_{2}$, we have
$$\sum_{n\geq0}s_{n}\left(|x+y|^{2}\right)\leq 2\left[\sum_{n\geq0}s_{n}\left(|x|^{2}\right)+\sum_{n\geq0}s_{n}\left(|y|^{2}\right)\right].$$

$(3)$ (H$\ddot{o}$lder Inequality of $S_{p}$): If $x\in S_{p}, \ y\in S_{q}$, where $\frac{1}{p}+\frac{1}{q}=1$ with $p\geq1$,  we have $${\rm Tr}(|xy|)\leq \|x\|_{p}\|y\|_{q}$$ or
\begin{eqnarray*}
\sum_{n=1}^{\infty}s_{n}(|xy|)&\leq&\left(\sum_{n=1}^{\infty}s^{p}_{n}(|x|)\right)^{\frac{1}{p}}\left(\sum_{n=1}^{\infty}s^{q}_{n}(|y|)\right)^{\frac{1}{q}}\\
&=&\left({\rm Tr}\left(|x|^{p}\right)\right)^{\frac{1}{p}}\left({\rm Tr}\left(|x|^{q}\right)\right)^{\frac{1}{q}}.
\end{eqnarray*}
\begin{proof}
First, if $x\in S_{\varphi}$, similar with the theorem 1.29 of \cite{Chen} we know, if there exists $k>0$ such that $\sum\limits_{n=1}^{\infty}\psi(h(s_{n}(k|x|)))=1$, then $\|x\|_{\varphi}^{o}=\frac{1}{k}\left[1+{\rm Tr}(\varphi(k(|x|)))\right]$, where $\psi$ is the  complementary function of $\varphi$ such that $\varphi(x)=\int_{0}^{|x|}h(t){\rm d}t$.

Now, if $x\in S_{p}$ which means $\varphi(x)=|x|^{p}$, then $h(t)=pt^{p-1}$ and $\psi(y)=cy^{q}$, where $c=\frac{1}{p^{\frac{p}{q}}}\cdot\frac{1}{q}$. If
$\sum\limits_{n=1}^{\infty}\psi(h(s_{n}(k|x|)))=1$, then
\begin{eqnarray*}
\sum\limits_{n=1}^{\infty}\psi(pk^{p-1}(s^{p-1}_{n}(|x|))&=& \sum\limits_{n=1}^{\infty}cp^{q}k^{(pq-q)}(s^{(pq-q)}_{n}(|x|))\\
&=&\sum\limits_{n=1}^{\infty}\frac{1}{p^{\frac{q}{p}}}\frac{1}{q}p^{q}k^{p}(s^{p}_{n}(|x|))\\
&=&\sum\limits_{n=1}^{\infty}\frac{p}{q}k^{p}(s^{p}_{n}(|x|))\\
&=&\sum\limits_{n=1}^{\infty}(p-1)k^{p}(s^{p}_{n}(|x|))\\
&=&1
\end{eqnarray*}
hence, $$k^{p}=\frac{1}{(p-1)\sum\limits_{n=1}^{\infty}(s^{p}_{n}(|x|))}$$
or $$k=\frac{1}{(p-1)^{\frac{1}{p}}\|x\|_{p}}.$$
Now one can get
\begin{eqnarray*}
\|x\|_{\varphi}^{o}&=&\frac{1}{k}\left[1+{\rm Tr}(\varphi(k(|x|)))\right]\\
&=&(p-1)^{\frac{1}{p}}\|x\|_{p}\left[1+\sum\limits_{n=1}^{\infty}k^{p}(s^{p}_{n}(|x|))\right]\\
&=&(p-1)^{\frac{1}{p}}\|x\|_{p}\left[1+\frac{\sum\limits_{n=1}^{\infty}(s^{p}_{n}(|x|))}{(p-1)\sum\limits_{n=1}^{\infty}(s^{p}_{n}(|x|))}\right]\\
&=&p(p-1)^{\frac{1}{p}-1}\|x\|_{p}\\
&=&\frac{p}{(p-1)^{\frac{1}{q}}}\|x\|_{p}.\\
\end{eqnarray*}
On the other hand,
\begin{eqnarray*}
\|y\|_{\varphi}&=&\inf\left\{\lambda\geq0, \ \ {\rm Tr}\left(\psi\left(\frac{y}{\lambda}\right)\right)\leq1\right\}\\
&=&\inf\left\{\lambda\geq0, \ \ \sum\limits_{n=1}^{\infty}\left(\psi\left(\frac{s_{n}(|y|)}{\lambda}\right)\right)\leq1\right\}\\
&=&\inf\left\{\lambda\geq0, \ \ \sum\limits_{n=1}^{\infty}c\left(\frac{s_{n}^{q}(|y|)}{\lambda^{q}}\right)\leq1\right\}\\
&=&\inf\left\{\lambda\geq0, \ \ c^{\frac{1}{q}}\sum\limits_{n=1}^{\infty}\left(s_{n}^{q}(|y|)\right)^{\frac{1}{q}}\leq\lambda\right\}\\
&=&\frac{1}{p^{\frac{1}{p}}}\frac{1}{q^{\frac{1}{q}}}\|y\|_{q}.
\end{eqnarray*}
Finally, by (4) of Lemma 3.1, we get
$${\rm Tr}(|xy|)\leq \|x\|_{\varphi}^{o}\|y\|_{\varphi}=\frac{p}{(p-1)^{\frac{1}{q}}}\|x\|_{p}\cdot\frac{1}{p^{\frac{1}{p}}}\frac{1}{q^{\frac{1}{q}}}\|y\|_{q}=\|x\|_{p}\|y\|_{q},$$
or
$$\sum_{n=1}^{\infty}s_{n}(|xy|)\leq\left(\sum_{n=1}^{\infty}s^{p}_{n}(|x|)\right)^{\frac{1}{p}}\left(\sum_{n=1}^{\infty}s^{q}_{n}(|y|)\right)^{\frac{1}{q}}.$$
In particular, for $p=2$, we get the Schwarz inequality 
$${\rm Tr}(|xy|)\leq \|x\|_{2}\|y\|_{2},$$
or
$$\sum_{n=1}^{\infty}s_{n}(|xy|)\leq\left(\sum_{n=1}^{\infty}s^{2}_{n}(|x|)\right)^{\frac{1}{2}}\left(\sum_{n=1}^{\infty}s^{2}_{n}(|y|)\right)^{\frac{1}{2}}.$$
\end{proof}
\begin{remark}
$(i)$ For above corollary, if $p=1$ then $q=\infty$, we also have $${\rm Tr}(|xy|)\leq \|x\|_{1}\|y\|_{\infty}={\rm Tr}(|x|)\|y\|_{\infty},$$ and $${\rm Tr}(|yx|)\leq \|y\|_{1}\|x\|_{\infty}={\rm Tr}(|y|)\|x\|_{\infty},$$ where $\|\cdot\|_{\infty}$ is the operator norm, this means that $S_{1}$ is a two-sides ideal of $\mathcal{B}(\mathbb{H})$.

$(ii)$ This corollary is just be the H$\ddot{o}$lder Inequality has been obtained by Ruskai in \cite{Ruskai}.
\end{remark}

$(4)$ If $\varphi(x)$ is invertible on $\mathbb{R}^{+}$, and $\varphi(x)\in \delta_{2}(0)$, then $$\|e_{n}\otimes e_{n}\|_{\varphi}=\frac{1}{\varphi^{-1}(1)}.$$
\begin{proof}
Since $\|e_{n}\otimes e_{n}\|=1$ and $e_{n}\otimes e_{n}$ is a one dimensional compact operator , where $\|\cdot\|$  is the operator norm, then $s_{n}(e_{n}\otimes e_{n})=1$ for any $n=1,2,\cdots$. Hence, by (1) of theorem 3.1
$$1={\rm Tr}\left(\varphi\left(\frac{e_{n}\otimes e_{n}}{\|e_{n}\otimes e_{n}\|_{\varphi}}\right)\right)=\varphi\left(s_{n}\left(\frac{e_{n}\otimes e_{n}}{\|e_{n}\otimes e_{n}\|_{\varphi}}\right)\right)=\varphi\left(\frac{1}{\|e_{n}\otimes e_{n}\|_{\varphi}}\right),$$
hence we can get the conclusion.
\end{proof}
\begin{example}
In \cite{Streater}, in order to study entropy in information geometry the author consider the quantum Orlicz function $\varphi(x)=coshx-1$ with $x\geq1$. It is obvious  a Orlicz function and $\varphi\in \delta_{2}(0)$, then we have $\|e_{n}\otimes e_{n}\|_{\varphi}=\frac{1}{\ln(2+\sqrt{3})}.$
In fact, since $\varphi(x)=coshx-1$ we have $\varphi^{-1}(1)=\ln(2+\sqrt{3})$, by (3) of corollary 3.1 we can get the conclusion.
\end{example}
\end{corollary}
\begin{theorem}
If $x\in \mathcal{B}(\mathbb{H})$ is a positive compact operator, then $x\in S_{\varphi}$ if and only if $\varphi(x)\in S_{1}$.
\end{theorem}
\begin{proof}
Since $x$ is a positive compact operator, then
$$\varphi(x)\eta=\sum^{\infty}_{n=1}\varphi(\lambda_{n})\langle\eta,e_{n}\rangle e_{n},\ \ \eta\in \mathbb{H}$$
where $\{0\}\cup\{\varphi(\lambda_{n})\}$ is spectral set. By the spectral mapping theorem, $\{\varphi(\lambda_{n})\}=\sigma(\varphi(x))\backslash\{0\}$.
Since $\varphi$ is nondecreasing and $\varphi(\lambda_{n})$ is the singular value sequence of $\varphi(x)$ we can get the conclusion.
\end{proof}
\begin{theorem}
For any $x\in S_{\varphi}$, we have

$(1)$ $\|x\|_{\varphi}=\||x|\|_{\varphi}=\|x^{\ast}\|_{\varphi};$

$(2)$ $S_{\varphi}$ is a $*$-two sides ideal of $\mathcal{B}(\mathbb{H})$, that means for any $y,z \in \mathcal{B}(\mathbb{H})$, then one have $yxz, zxy \in S_{\varphi}$ and $$\|yxz\|_{\varphi}\leq \|y\|_{\infty}\|x\|_{\varphi}\|z\|_{\infty}.$$
\end{theorem}
\begin{proof}
(1) By the definition of the $\|\cdot\|_{\varphi}$ we can get $\|x\|_{\varphi}=\||x|\|_{\varphi}.$
Let $x=u|x|$ be the polar decomposition of $x$. Then for any $\lambda>0$ we have $u|\frac{x}{\lambda}|u^{\ast}=|\frac{x^{\ast}}{\lambda}|$, by the induction we get $u|\frac{x}{\lambda}|^{n}u^{\ast}=|\frac{x^{\ast}}{\lambda}|^{n}$, hence for any polynomial $P$ we have
$uP\left(|\frac{x}{\lambda}|\right)u^{\ast}=P\left(|\frac{x^{\ast}}{\lambda}|\right)$. Then for any continuous function $\varphi$, one get $u\varphi\left(|\frac{x}{\lambda}|\right)u^{\ast}=\varphi\left(|\frac{x^{\ast}}{\lambda}|\right)$.
By the property of ${\rm Tr}$ we get $${\rm Tr}\left(u\varphi\left(\left|\frac{x}{\lambda}\right|\right)u^{\ast}\right)={\rm Tr}\left(\varphi\left(\left|\frac{x}{\lambda}\right|\right)\right)={\rm Tr}\left(\varphi\left(\left|\frac{x^{\ast}}{\lambda}\right|\right)\right)$$
which can get the conclusion.

$(2)$ If $x\in S_{\varphi}$, then there exists a $\lambda_{0}>0$ such that
$${\rm Tr}(\varphi(\lambda_{0}x))=\sum\limits_{n\geq0}\varphi\left(\lambda_{0}s_{n}(|x|)\right)<\infty.$$
By the inequality 2.3 of \cite{GK} and the nondecreasing property of $\varphi$, for any $y \in \mathcal{B}(\mathbb{H})$ and let $\lambda_{1}=\frac{\lambda_{0}}{\|y\|_{\infty}}>0$,
\begin{eqnarray*}
{\rm Tr}(\varphi(\lambda_{1}xy))&=&\sum\limits_{n\geq0}\varphi\left(\lambda_{1}s_{n}(|xy|)\right)\\
&\leq&\sum\limits_{n\geq0}\varphi\left(\lambda_{1}s_{n}(|x|)\|y\|_{\infty}\right)\\
&=&\sum\limits_{n\geq0}\varphi\left(\lambda_{0}s_{n}(|x|)\right)\\
&<&\infty.
\end{eqnarray*}
On the other hand, combine  the inequality 2.3 of \cite{GK},  the nondecreasing property of $\varphi$ with the (2) of Lemma 3.1, we have
$$\sum\limits_{n\geq0}\varphi\left(\frac{s_{n}(|xy|)}{\|x\|_{\varphi}\|y\|_{\infty}}\right)\leq\sum\limits_{n\geq0}\varphi\left(\frac{s_{n}(|x|)\|y\|_{\infty}}{\|x\|_{\varphi}\|y\|_{\infty}}\right)\leq1,$$
by the definition of the Luxemburg we get $$\|xy\|_{\varphi}\leq\|x\|_{\varphi}\|y\|_{\infty}.$$
Hence, one get that $S_{\varphi}$ is a right ideal of $\mathcal{B}(\mathbb{H})$, furthermore using the equality 2.1 of \cite{GK} $(s_{n}(x)=s_{n}(x^{*}))$ one get $S_{\varphi}$ is a $*$-right ideal of $\mathcal{B}(\mathbb{H})$ and
\begin{equation}
\|xy\|_{\varphi}\leq\|x\|_{\varphi}\|y\|_{\infty}.
\end{equation}
Next, by the equality 2.1 of \cite{GK} and using the similar methods above we get $S_{\varphi}$ is a $*$-left ideal of $\mathcal{B}(\mathbb{H})$ and \begin{equation}
\|yx\|_{\varphi}\leq\|y\|_{\infty}\|x\|_{\varphi}.
\end{equation}
At last, for any $y,z\in \mathcal{B}(\mathbb{H})$, since $yx\in S_{\varphi}$ and $xz\in S_{\varphi}$  hence $yxz\in S_{\varphi}$ with $$\|yxz\|_{\varphi}\leq \|y\|_{\infty}\|x\|_{\varphi}\|z\|_{\infty}$$
using (1) and (2) one can complete the proof.
\end{proof}

\begin{theorem}
 $(E_{\varphi})^{*}=S^{o}_{\psi}$ in the sense that for any $(E_{\varphi})^{*}$, there is a unique $y\in S^{o}_{\psi}$ such that
 $$T_{y}(x)={\rm Tr}(xy)  \ \ (x\in (E_{\varphi})^{*})$$
 and the mapping $T_{y}\mapsto y$ is isometric from $(E_{\varphi})^{*}$ to
 $S^{o}_{\psi}$.
\end{theorem}
\begin{proof}
By [4] of Lemma 3.1, one get that $|T_{y}(x)|=|{\rm Tr}(xy)|\leq \|y\|^{o}_{\psi}\|x\|_{\varphi}$ and $\|T_{y}\|\leq \|y\|_{\psi}^{o}$.

Now, we define $\Phi: S^{o}_{\psi}\rightarrow (E_{\varphi})^{*}$ with $\Phi(y)=T_{y}$, then  $\Phi$ is a linear bounded operator. For any $\varepsilon>0$, find a unit vector $e\in \mathbb{H}$ such that $\|ye\|_{\psi}^{o}\geq\|y\|_{\psi}^{o}-\varepsilon$. Find another unit vector $h\in \mathbb{H}$ such that $\|ye\|_{\psi}^{o}=\langle ye,h\rangle={\rm Tr}(y(e\otimes h))$, hence $$T_{y}(e\otimes h)={\rm Tr}(y(e\otimes h))=\langle ye,h\rangle=\|y\|_{\psi}^{o}-\varepsilon.$$
 It is obvious $\|e\otimes h\|_{\varphi}=\|e\|\|f\|=1$. Hence, one get $\|T_{y}\|\geq\|y\|_{\psi}^{o}$ which implies that $\|T_{y}\|=\|y\|_{\psi}^{o}$. So, the mapping $\Phi :S^{o}_{\psi}\mapsto(E_{\varphi})^{*} $ is isometric.

Now, let $L\in (E_{\varphi})^{*}$ and easy to verify $L(e\otimes h)$ is a bounded bilinear functional. Because
$$|L(e\otimes h)|\leq \|L\|\|e\otimes h\|=\|L\|\|e\|\|h\|.$$
By Riesz theorem, there exist a $y\in S_{\psi}$  such that $L(e\otimes h)=\langle ye,h\rangle={\rm Tr}(y(e\otimes h))$. Hence for any finite rank operator $f$, we have $L(f)={\rm Tr}(yf)$. Since finite rank operator is dense on $E_{\varphi}$, then $L(x)={\rm Tr}(xy)=T_{y}(x)$ for any $x\in E_{\varphi}$, which means $L=T_{y}$ and $L: S^{o}_{\psi}\mapsto(E_{\varphi})^{*}$ is an isometric isomorphism which can get the conclusion.
\end{proof}
\begin{remark}
Using the similar method we also can get $S_{\varphi}=(E^{o}_{\psi})^{*}$.
\end{remark}
Combing the theorem 2.1, 3.4 and remark 1, one can get the following corollary,
\begin{corollary}
If $\varphi\in \delta_{2}(0)$ and $\psi\in \delta_{2}(0)$, $S_{\varphi}$ is reflexive. Especially, the $S_{p}$ is reflexive for $p>1$.
\end{corollary}
\section{Application}
The Bergman space $L_{a}^{2}(\mathbb{D})$ on open unit disk $\mathbb{D}=\{z\in \mathbb{C}: |z|<1\}$ as following
$$L_{a}^{2}(\mathbb{D})=\left\{f(z)\ is\ analytic \  functions \  on\ \mathbb{D}: \|f\|^{2}=\frac{1}{\pi}\int_{\mathbb{D}}|f(z)|^{2}{\rm d}A(z)<\infty\right\},$$
where ${\rm d}A(z)$ is area metric. Then $L_{a}^{2}(\mathbb{D})$ is a Hilbert space which has a orthonormal basis $e_{n}(z)=\sqrt{n+1}z^{n},  n=0,1,2,\cdots.$ From the Example 5.17 of \cite{guo} we know, the Toeplitz operator $x_{1-|z|^{2}}$ on Bergman space is a compact self-adjoint operator and $x_{1-|z|^{2}}=\bigoplus\limits^{\infty}_{n=0}\frac{1}{n+2}e_{n}\otimes e_{n}$ is the spectral decomposition.  Now we show $x_{1-|z|^{2}}$ belongs to some $S_{\varphi}$ .

In fact, one could find some Orlicz functions $\varphi$ with $\sum\limits_{n\geq 0}\varphi\left(\frac{1}{n+1}\right)<\infty$, then $$\varphi\left(x_{1-|z|^{2}}\right)=\bigoplus\limits^{\infty}_{n=0}\varphi\left(\frac{1}{n+2}\right)e_{n}\otimes e_{n}$$ and $${\rm Tr}\left(\varphi\left(x_{1-|z|^{2}}\right)\right)=\sum\limits_{n\geq 0}\varphi\left(\frac{1}{n+1}\right)<\infty.$$
Hence, $x_{1-|z|^{2}}\in S_{\varphi}$. It is obvious, $x_{1-|z|^{2}}$ not belongs to $S_{1}$ but $S_{p}, \ p>1$.

Generally speaking, it is difficulty to calculate the norm of the $\|x_{1-|z|^{2}}\|_{\varphi}$ since the Orlicz function is abstract. But, if $\varphi\in \delta_{2}(0)$, by [1] of theorem 3.1, the norm of $x_{1-|z|^{2}}$ such that $$1=\sum\limits_{n\geq0}\varphi\left(\frac{1}{(n+2)\|x_{1-|z|^{2}}\|_{\varphi}}\right).$$ Especially, if $\varphi=|x|^{p}, \ p>1$, one can get $$\|x_{1-|z|^{2}}\|_{p}=\left[\sum\limits_{n\geq0}\frac{1}{\left(n+2\right)^{p}}\right]^{\frac{1}{p}}=\left(\zeta(p)-1\right)^{\frac{1}{p}},$$
where $\zeta(p)$ is the Riemann function.

\section*{References}

\bibliography{mybibfile}

\end{document}